\documentclass[reqno,12pt]{amsart}
\usepackage{amsmath,amsfonts,amsthm,amssymb,indentfirst}
\usepackage[all,poly]{xy} 
\usepackage{color}
\usepackage[pagebackref,colorlinks]{hyperref}
\usepackage{kantlipsum}
\usepackage[all,poly]{xy}
\usepackage{mathrsfs}
\usepackage{enumerate}

\theoremstyle{plain}
\newtheorem{lemma}{Lemma}[section] 
\newtheorem{theorem}[lemma]{Theorem}
\newtheorem{corollary}[lemma]{Corollary}

\theoremstyle{definition}

\newtheorem{example}[lemma]{Example}

\newcommand{\Zset}{\mathbb Z}

\newcommand{\gr}{\operatorname{gr}}

\newcommand{\so}{\mathbf{s}}
\newcommand{\ra}{\mathbf{r}}

\newcommand{\V}{\mathcal V}

\newcommand{\POG}{\mathbf {POG}}
\newcommand{\LE}{\mathbf {LE_{fin}}}

\newcommand\sima{\sim}

\title[Fullness and weak faithfulness]{The functor $K_0^{\operatorname{gr}}$ is full and only weakly faithful}

\author{Lia Va\v s}
\address{Department of Mathematics, Saint Joseph's University, Philadelphia, PA 19131, USA}
\email{lvas@sju.edu}



\begin{document}

\begin{abstract}
The Graded Classification Conjecture states that the pointed $K_0^{\operatorname{gr}}$-group is a complete invariant of the Leavitt path algebras of finite graphs when these algebras are considered with their natural grading by $\mathbb Z.$ The strong version of this conjecture states that the functor $K_0^{\operatorname{gr}}$ is full and faithful when considered on the category of Leavitt path algebras of finite graphs and their graded homomorphisms modulo conjugations by invertible elements of the zero components. We show that the functor $K_0^{\operatorname{gr}}$ is full for the unital Leavitt path algebras of countable graphs and that it is faithful (modulo specified conjugations) only in a certain weaker sense.
\end{abstract}

\maketitle

\section{Introduction}

If $E$ is a directed graph and $K$ a field, the Leavitt path algebra $L_K(E)$ and its operator theory counterpart, the graph $C^*$-algebra $C^*(E),$ have been the subjects of a variety of important results as well as of some thought provoking conjectures. We focus on one of these conjectures, the Graded Classification Conjecture, and its stronger version.

The algebra $L_K(E)$ is naturally graded by the group of integers $\Zset.$ In the unital case (when $E$ has finitely many vertices), this grading induces an action of the infinite cyclic group $\Gamma=\langle t\rangle\cong \Zset$ on the set of the graded isomorphism classes of finitely generated graded projective $L_K(E)$-modules. This action makes the Grothendieck group formed using the finitely generated {\em graded} projective modules and their {\em graded} isomorphism classes into a pre-ordered $\Gamma$-group. Although the notation $K_0^{\operatorname{gr}}(L_K(E))$ has often been used for this group, we use $K_0^\Gamma(L_K(E))$ in order to emphasize that the Grothendieck group itself is not graded by $\Gamma$ but that $\Gamma$ acts on it. If $L_K(E)$ is unital, $[L_K(E)]$ is an order-unit of the group $K_0^\Gamma(L_K(E)).$ If $L_K(E)$ is not unital, $K_0^\Gamma(L_K(E))$ can be defined via an unitization of $L_K(E)$ and a certain generating interval can be considered instead of $[L_K(E)].$  

The Graded Classification Conjecture was formulated by Roozbeh Hazrat in \cite{Roozbeh_Annalen} (published in 2013, on arXiv since 2011). The conjecture was originally stated for row-finite graphs. However, without considering generating intervals, the conjecture is applicable only to the unital case.
The statement below matches a widely accepted formulation (e.g. in \cite{LPA_book} and in  \cite{Ara_Pardo_graded_K_classification}).

{\bf The Graded Classification Conjecture (GCC)} is stating that the following conditions are equivalent for finite graphs $E$ and $F$ and a field $K.$ 
\begin{enumerate}
\item The algebras $L_K(E)$ and $L_K(F)$ are isomorphic as unital $\Zset$-graded $K$-algebras. 
\item There is an order-preserving $\Gamma$-group isomorphism of $K_0^\Gamma(L_K(E))$ and $K_0^\Gamma(L_K(F))$ which maps the order-unit $[L_K(E)]$ to the order-unit $[L_K(F)].$ 
\end{enumerate}

Let $\LE$ stand for the category of Leavitt path algebras of finite graphs in which the graded homomorphisms are considered modulo conjugations by invertible elements of the zero components of the algebras. Let $\POG^u$ stand for the category of pre-ordered $\Gamma$-groups with order-units and $\Zset[\Gamma]$-module homomorphisms which preserve the order and the order-units. If we consider $K_0^\Gamma$ as a functor from $\LE$ to $\POG^u$, then {\bf the Strong Graded Classification Conjecture}, also formulated in \cite{Roozbeh_Annalen}, is stating that the functor $K_0^\Gamma$ is full and faithful.

In \cite{Roozbeh_Annalen}, the GCC is shown to hold for polycephalic graphs. These are finite graphs in which every vertex connects to a sink, a cycle without exits, or to a vertex emitting no other edges but finitely many loops  and the graph is such that when these loops as well as an edge of each cycle with no exits are removed, the resulting graph is a finite acyclic graph. In \cite{Ara_Pardo_graded_K_classification}, a weaker version of the GCC is shown to hold for finite graphs without sources or sinks. In \cite{Roozbeh_Lia_Ultramatricial}, the GCC, generalized to include the non-unital case, is shown to hold for countable graphs in which no cycle has an exit and in which every infinite path ends in a (finite or infinite) sink or in a cycle. By \cite{Eilers_et_al}, the generalized GCC holds for countable graphs for which whenever there is an edge from a vertex $v$ to a vertex $w,$ there are infinitely many edges from $v$ to $w$. 

If $\V^\Gamma(L_K(E))$ is the $\Gamma$-monoid of the graded isomorphism classes of finitely generated graded projective modules (or its non-unital version defined using an unitization of $L_K(E)$ in the case when $L_K(E)$ is non-unital), this monoid is {\em cancellative} by \cite[Corollary 5.8]{Ara_et_al_Steinberg}. Since we prominently use this result, we  illustrate why this cancellability holds while the standard (``nongraded'') monoid $\V(L_K(E))$ can be fatally non-cancellative. Consider the graph $\;\;\xymatrix{\bullet\ar@(ul,dl)_e\ar@(ur, dr)^f}\;\;$ and let $R$ stand for its Leavitt path algebra over any field $K.$
The map $(x, y)\mapsto xe+yf$ is a left $R$-module isomorphism of $R\oplus R$ and $R.$ The existence of such map implies that the relation 
\[[R]=[R]+[R]\]
holds in $\V(R)$ and forces cancellability to fail on $\V(R).$ However, in the graded case, this map is an isomorphism of $R(1)\oplus R(1),$ not $R\oplus R,$ and $R$  
where $R(1)$ is $R$ shifted by $1\in \Zset$ (more detail are reviewed in section \ref{subsection_graded_rings}). Thus, the relation \[[R]=t[R]+t[R]\]
holds in $\V^\Gamma(L_K(E)),$ producing no obstruction of cancellability. 
Cancellability enables one to go back and forth from an equality of positive elements at the $K_0^\Gamma$-level to a graded isomorphism on the algebra level (Lemma \ref{lemma_cancellativity}). Thus, a Leavitt path algebra shares this favorable property with a graded ultramatricial algebra over a graded field. In the analogy, the vertices correspond to some of the diagonal standard matrix units and the edges to some of the off-diagonal standard matrix units. This analogy is the underlying idea of our proof of Theorem \ref{theorem_fullness}, stating that 
\begin{center}
the functor $K_0^\Gamma$ is full for Leavitt path algebras of countable graphs with finitely many vertices.  
\end{center}
In particular, $K_0^\Gamma$ {\em is} full on $\LE.$ %
Theorem \ref{theorem_fullness} also states that
if a map on the $K_0^\Gamma$-level is 
{\em injective}, then one can obtain an {\em injective} graded algebra homomorphism on the algebra level. Thus, if condition (2) holds for two finite graphs $E$ and $F,$ then there are graded algebra embeddings
$L_K(E)\hookrightarrow L_K(F)$ and $L_K(F)\hookrightarrow L_K(E)$ which
are mutually inverse isomorphisms on the $K_0^\Gamma$-level (Corollary \ref{corollary_of_fullness}). 

In section \ref{section_faithfullness}, we show Theorem \ref{theorem_faithfulness} stating that $K_0^\Gamma$ is faithful for finite graphs in the following weaker sense: for graded homomorphisms $\phi,\psi: L_K(E)\to L_K(F),$ $K_0^\Gamma(\phi)=K_0^\Gamma(\psi)$
if and only if
\begin{enumerate}
\item[(LE)] there is a zero-degree invertible element $z\in L_K(F)$ such that $\phi(v)=z\psi(v)z^{-1}$ for every vertex $v$ of $E$.
\end{enumerate}
Example \ref{example_not_faithful} shows that one cannot find an element $z$ as in (LE) such that $\phi(e)=z\psi(e)z^{-1}$ also holds for all edges $e$ of $E$. By this example (and also by \cite[Example 6.7]{Ara_Pardo_graded_K_classification}), 
$K_0^\Gamma$ is not faithful on $\LE$ so the Strong GCC fails. However, if we say that two graded homomorphisms are {\em locally equal} if  (LE) holds for them, then the functor $K_0^\Gamma$ is full and faithful on the category obtained from $\LE$ by considering graded homomorphisms modulo local equality instead of modulo conjugations by invertible elements of the zero components.

Soon after submitting the present paper for publication, the author became aware of Arnone's work \cite{Arnone} which overlaps with the present results. The two sets of results, obtained independently and using different methods, appeared on arXiv almost simultaneously (the first version of \cite{Arnone} was submitted on June 14, 2022 and the first version of the present paper on June 13, 2022). We summarize the similarities and differences of the main results.  

The fullness result \cite[Theorem 6.1]{Arnone} is  comparable to Theorem \ref{theorem_fullness}. The algebras in Theorem \ref{theorem_fullness} are considered over a field while they are considered over a commutative unital ring in \cite[Theorem 6.1]{Arnone}. The graded algebra map obtained in 
\cite[Theorem 6.1]{Arnone} is also involution and diagonal  preserving. On the other hand, the graphs in \cite[Theorem 6.1]{Arnone} are assumed to be finite and they can have countably infinitely many edges in Theorem \ref{theorem_fullness}. The level of constructiveness (i.e. whether one can explicitly obtain an algebra homomorphism given a $K_0^\Gamma$-level map) of \cite[Theorem 6.1]{Arnone} is discussed at the end of \cite[Section 1]{Arnone} and we discuss this level for Theorem \ref{theorem_fullness} in section \ref{subsection_constructiveness}.

We also note that \cite[Corollary 3.5]{Arnone} is comparable to Theorem \ref{theorem_faithfulness}.   \cite[Corollary 3.5]{Arnone} lists a condition which is seemingly stronger than the condition (LE) but, by the arguments of the proofs of Theorem \ref{theorem_faithfulness} and \cite[Corollary 3.5]{Arnone}, it is equivalent to (LE).

\section{Prerequisites}
\label{section_prerequisites}

In sections \ref{subsection_preordered_groups} to \ref{subsection_graded_Grothendieck}, $\Gamma$ stands for arbitrary group and $\varepsilon$ is the identity of $\Gamma$. In the rest of the paper, $\Gamma$ stands for the infinite cyclic group $\langle t\rangle$ on one generator $t.$ 

\subsection{Pre-ordered \texorpdfstring{$\Gamma$}{TEXT}-monoids and \texorpdfstring{$\Gamma$}{TEXT}-groups}\label{subsection_preordered_groups}
If $M$ is an abelian monoid with a left action of a group $\Gamma$ which is compatible with the monoid operation, we say that $M$ is a {\em $\Gamma$-monoid}. If $G$ an abelian group with a left action of $\Gamma$ which agrees with the group operation, we say that $G$ is a {\em $\Gamma$-group}. Such action of $\Gamma$ uniquely determines a left $\Zset[\Gamma]$-module structure on $G,$ so $G$ is also a left $\Zset[\Gamma]$-module.  

Let $\geq$ be a reflexive and transitive relation (a pre-order) on a $\Gamma$-monoid $M$ ($\Gamma$-group $G$) such that $g_1\geq g_2$ implies $g_1 + h\geq g_2 + h$ and $\gamma g_1 \geq \gamma g_2$ for all $g_1, g_2, h$ in $M$ (in $G$) and $\gamma\in \Gamma.$ We say that such monoid $M$ is a {\em pre-ordered $\Gamma$-monoid} and that such a group $G$ is a {\em pre-ordered $\Gamma$-group}. 

If $G$ is a pre-ordered $\Gamma$-group, the $\Gamma$-monoid $G^+=\{x\in G\mid x\geq 0\}$ is the {\em positive cone} of $G.$ If $G$ and $H$ are pre-ordered $\Gamma$-groups, a $\Zset[\Gamma]$-module homomorphism $f\colon G\to H$ is {\em order-preserving} if $f(G^+)\subseteq H^+.$
An element $u$ of a pre-ordered $\Gamma$-group $G$ is an \emph{order-unit} if $u\in G^+$ and for any $x\in G$, there is a nonzero $a\in \Zset^+[\Gamma]$ such that $x\leq au.$ 
If $G$ and $H$ are pre-ordered $\Gamma$-groups with order-units $u$ and $v$ respectively, an order-preserving $\Zset[\Gamma]$-module homomorphism $f\colon G\to H$ is {\em order-unit-preserving} if $f(u)=v.$ One writes $f: (G, u)\to (H,v)$ in this case and says that $f$ is a homomorphism of pointed groups. 

We let $\POG_\Gamma$ denote the category whose objects are pre-ordered $\Gamma$-groups and whose morphisms are order-preserving $\Zset[\Gamma]$-homomorphisms, we let $\POG^u_\Gamma$ denote the category whose objects are pairs $(G, u)$ where $G$ is an object of $\POG_\Gamma$ and $u$ is an order-unit of $G$ and whose morphisms are morphisms of $\POG_\Gamma$ which are order-unit-preserving.

\subsection{Graded rings}\label{subsection_graded_rings}
A ring $R$ (not necessarily unital) is {\em graded} by a group $\Gamma$ if $R=\bigoplus_{\gamma\in\Gamma} R_\gamma$ for additive subgroups $R_\gamma$ and $R_\gamma R_\delta\subseteq R_{\gamma\delta}$ for all $\gamma,\delta\in\Gamma.$ The elements of the set $\bigcup_{\gamma\in\Gamma} R_\gamma$ are said to be {\em homogeneous}. A {\em graded} ring homomorphism $f:R\to S$ is a ring homomorphism such that $f(R_\gamma)\subseteq S_\gamma$ for $\gamma\in \Gamma.$ We write $\cong_{\gr}$ for a graded ring isomorphism. 
A unital and commutative graded ring $R$ is a \emph{graded field} if every nonzero homogeneous element has a multiplicative inverse. 

We use the standard definitions of graded right and left $R$-modules, graded module homomorphisms and we use $\cong_{\gr}$ for a graded module isomorphism. If $M$ is a graded right $R$-module and $\gamma\in\Gamma,$ the $\gamma$-\emph{shifted} graded right $R$-module $(\gamma)M$ is defined as the module $M$ with the $\Gamma$-grading given by $(\gamma)M_\delta = M_{\gamma\delta}$ for all $\delta\in \Gamma.$ If $N$ is a graded left module, the $\gamma$-shift of $N$ is the graded module $N$ with the $\Gamma$-grading given by $M(\gamma)_\delta = M_{\delta\gamma}$ for all $\delta\in \Gamma.$  

Any finitely generated graded free right $R$-module has the form $(\gamma_1)R\oplus\ldots\oplus (\gamma_n)R$ and any finitely generated graded free left $R$-module has the form
$R(\gamma_1)\oplus\ldots\oplus R(\gamma_n)$ for $\gamma_1, \ldots,\gamma_n\in\Gamma$ (\cite[Section 1.2.4]{Roozbeh_book} contains more details). A finitely generated graded projective module is a direct summand of a finitely generated graded free module.   

Analogously to the non-homogeneous idempotents, two homogeneous idempotents $p$ and $q$ of a graded ring $R$ are {\em orthogonal} if $pq=qp=0.$ In this case, $p+q$ is a homogeneous  idempotent. The relation $\leq,$ given by 
$p\leq q$ if $pq=qp=p$ for homogeneous idempotents $p$ and $q,$ is such that $p\leq q$ implies that $q-p$ is also a homogeneous idempotent and $p$ and $q-p$ are orthogonal.

\subsection{Graded algebraic equivalence}
\label{subsection_sima_and_sims}
Two idempotents $p,q$ of any ring $R$ are {\em algebraically equivalent}, written as $p\sima_R q$ (or simply $p\sima q$ if it is clear from context) if there are $x,y\in R$ such that $xy=p$ and $yx=q.$ In this case, we say that $x$ and $y$ {\em realize} the equivalence $p\sima q.$ By replacing $x$ with $pxq$ and $y$ with $qyp,$ we can assume that $x\in pRq$ and $y\in qRp$ so that the left multiplication with $x$ is an isomorphism $qR\cong pR$ with the left multiplication by $y$ as its inverse. Conversely, if $\phi: qR\cong pR$ for some idempotents $p,q\in R,$  then $\phi(q)=x$ and $\phi^{-1}(p)=y$ realize the relation $p\sima q.$ These claims generalize to graded rings as in the lemma below.  

\begin{lemma} \cite[Lemma 2.2]{Lia_graded_UR}
Let $R$ be a $\Gamma$-graded ring and $p,q$ homogeneous idempotents of $R$. The following conditions are equivalent. 
\begin{enumerate}[\upshape(1)]
\item $pR\cong_{\gr} (\gamma^{-1})qR$ for some $\gamma\in\Gamma.$
 
\item $Rp\cong_{\gr} Rq(\gamma)$ for some $\gamma\in\Gamma.$

\item There are $x\in R_{\gamma}$ and $y\in R_{\gamma^{-1}}$ such that $xy=p$ and $yx=q.$ 

\item There are $x\in pR_{\gamma}q$ and $y\in qR_{\gamma^{-1}}p$ such that $xy=p$ and $yx=q.$
\end{enumerate}
\label{lemma_graded_equivalence}
\end{lemma}
If $p$ and $q$ satisfy the conditions in the above lemma, we say that $p$ and $q$ are {\em graded algebraically equivalent} and write $p\sima_{\gr}q$. If $\gamma=\varepsilon,$ the relation $p\sima_{\gr}q$ is $p\sima_{R_\varepsilon} q.$ 

\subsection{The graded Grothendieck group}
\label{subsection_graded_Grothendieck}
If $R$ is a unital $\Gamma$-graded ring, let $\V^{\Gamma}(R)$ denote the monoid of the graded isomorphism classes $[P]$ of finitely generated graded projective right $R$-modules $P$ with the addition given by $[P]+[Q]=[P\oplus Q]$ and the left $\Gamma$-action given by $(\gamma, [P])\mapsto [(\gamma^{-1})P].$ The monoid $\V^\Gamma(R)$ can be represented using the classes of left modules in which case the $\Gamma$-action is $(\gamma, [P])\mapsto [P(\gamma)].$ The two representations are equivalent (see \cite[Section 1.2.3]{Roozbeh_book}). The group $K_0^{\Gamma}(R)$ can also be defined via homogeneous matrices and \cite[section 3.2]{Roozbeh_book} has more details. Although we focus on the case when $\Gamma$ is $\Zset,$ we note that the definitions and results of \cite[Section 3.2]{Roozbeh_book} carry to the case when $\Gamma$ is not necessarily abelian by \cite[Section 1.3]{Lia_realization}.
The \emph{Grothendieck $\Gamma$-group}  $K_0^{\Gamma}(R)$ is the group completion of the $\Gamma$-monoid $\V^{\Gamma}(R)$ with 
the action of $\Gamma$ inherited from $\V^{\Gamma}(R)$. If $\Gamma$ is the trivial group, $K_0^{\Gamma}(R)$ is the usual $K_0$-group.

The image of the $\Gamma$-monoid $\V^{\Gamma}(R)$ of a $\Gamma$-graded unital ring $R$ under the natural map $\V^{\Gamma}(R)\to K_0^{\Gamma}(R)$ is a positive cone making $K_0^{\Gamma}(R)$ into an object of $\POG_\Gamma.$ Moreover, $(K_0^{\Gamma}(R), [R])$ is an object of $\POG^u_\Gamma.$ If $\phi$ is a graded ring homomorphism, then $K_0^\Gamma(\phi)$ is a morphism of $\POG_\Gamma.$ If $\phi$ is unital (i.e. $\phi$ maps the identity onto identity), then $K_0^{\Gamma}(\phi)$ is a morphism of $\POG^u_\Gamma.$

\subsection{Graphs and Leavitt path algebras}\label{subsection_LPAs}
If $E$ is a directed graph, we let $E^0$ denote the set of vertices, $E^1$ denote the set of edges, and $\so$ and $\ra$ denote the source and the range maps of $E.$ A {\em sink} of $E$ is a vertex which emits no edges and an {\em infinite emitter} is a vertex which emits infinitely many edges. A vertex of $E$ is {\em regular} if it is not a sink nor an infinite emitter. The graph $E$ is {\em row-finite} if it has no infinite emitters, $E$ is {\em finite} if it has finitely many vertices and edges, and it is {\em countable} if it has countably many vertices and edges.  

A {\em path} is a single vertex or a sequence of edges $e_1e_2\ldots e_n$ for some positive integer $n$ such that $\ra(e_i)=\so(e_{i+1})$ for $i=1,\ldots, n-1.$ The functions $\so$ and $\ra$ extend to paths naturally. 

If $K$ is any field, the \emph{Leavitt path algebra} $L_K(E)$ of $E$ over $K$ is a free $K$-algebra generated by the set  $E^0\cup E^1\cup\{e^\ast\mid e\in E^1\}$ such that for all vertices $v,w$ and edges $e,f,$

\begin{tabular}{ll}
(V)  $vw =0$ if $v\neq w$ and $vv=v,$ & (E1)  $\so(e)e=e\ra(e)=e,$\\
(E2) $\ra(e)e^\ast=e^\ast\so(e)=e^\ast,$ & (CK1) $e^\ast f=0$ if $e\neq f$ and $e^\ast e=\ra(e),$\\
(CK2) $v=\sum_{e\in \so^{-1}(v)} ee^\ast$ for each regular vertex $v.$ &\\
\end{tabular}

By the first four axioms, each element of $L_K(E)$ is a sum of the form $\sum_{i=1}^n k_ip_iq_i^\ast$ for some $n$, paths $p_i$ and $q_i$, and elements $k_i\in K,$ for $i=1,\ldots,n$ where $v^*=v$ for $v\in E^0$ and $p^*=e_n^*\ldots e_1^*$ for a path $p=e_1\ldots e_n.$  The algebra $L_K(E)$ is unital if and only if $E^0$ is finite in which case the sum of all vertices is the identity. If we consider $K$ to be trivially graded by $\Zset,$ $L_K(E)$ is graded by $\Zset$ so that the $n$-component $L_K(E)_n$ is the $K$-linear span of the elements $pq^\ast$ for paths $p, q$ with $|p|-|q|=n$ where $|p|$ denotes the length of a path $p.$ 
By \cite[Proposition 2.1.14 and Corollary 2.1.16]{LPA_book}, the zero component $L_K(E)_0$ is isomorphic to a direct limit of matricial algebras. If $E$ is row-finite, the direct limit can be taken over $\Zset^+$ so that $L_K(E)_0$ is isomorphic to an ultramatricial algebra and, if $E$ is finite, the connecting maps are unital. 

If $R$ is a $K$-algebra which contains the elements $p_v$ for $v\in E^0,$ and $x_e$ and $y_e$ for $e\in E^1$ such that the five axioms hold for these elements, the set of these elements forms an {\em $E$-family}. In this case, the Universal Property of $L_K(E)$ states that there is a unique algebra homomorphism $\phi:L_K(E)\to R$ such that $\phi(v)=p_v, \phi(e)=x_e,$ and $\phi(e^*)=y_e$ (see \cite[Remark 1.2.5]{LPA_book}). By the Graded Uniqueness Theorem (\cite[Theorem 2.2.15]{LPA_book}), $\phi$ is injective if $p_v\neq 0$ for $v\in E^0.$ If $R$ is $\Zset$-graded and $p_v\in R_0$ for $v\in E^0,$ $x_e\in R_1$ and $y_e\in R_{-1}$ for $e\in E^1,$ then $\phi$ is graded.

\subsection{The Grothendieck \texorpdfstring{$\Gamma$}{TEXT}-group of a graph}\label{subsection_graph_group}

If $E$ is a graph and $\Gamma$ is a group, the authors of \cite{Ara_et_al_Steinberg} construct a commutative monoid $M^\Gamma_E$ which is isomorphic to $\V^{\Gamma}(L_K(E))$ as a $\Gamma$-monoid. The authors of \cite{Roozbeh_Lia_comparability} provide an alternative construction of $M^\Gamma_E$ which we briefly review. 

Recall that we fixed $\Gamma=\langle t\rangle$ to be the infinite cyclic group on $t$ from after section \ref{subsection_graded_Grothendieck} on. The {\em graph $\Gamma$-monoid $M_E^\Gamma$} (also called the talented monoid) is the free abelian $\Gamma$-monoid on generators $[v]$ for $v\in E^0$ and $[v-\sum_{e\in Z}ee^*]$ for infinite emitters $v$ and nonempty and finite sets $Z\subseteq \so^{-1}(v)$ subject to the relations
\[[v]=\sum_{e\in \so^{-1}(v)}t[\ra(e)],\hskip.4cm [v]=[v-\sum_{e\in Z}ee^*]+\sum_{e\in Z}t[\ra(e)],\,\mbox{ and }\;[v-\sum_{e\in Z}ee^*]=[v-\sum_{e\in W}ee^*]+\sum_{e\in W-Z}t[\ra(e)]\]
where $v$ is a vertex which is regular for the first relation and an infinite emitter for the second two relations in which $Z\subsetneq W$ are finite and nonempty subsets of $\so^{-1}(v).$ The power 1 of $t$ in the formulas above signifies the fact that $e$ is a path of unit length. Consequently, if $p$ is a path of length $n$, the relation $[\so(p)]=t^n[\ra(p)]+a$ holds in $M_E^\Gamma$ for some $a\in M_E^\Gamma.$ Thus, the ``talent'' of this monoid is to register the lengths of paths between vertices while the standard monoid $\V(L_K(E))$ can only register whether two vertices are connected or not.  

The monoids $M^\Gamma_E$ and $\V^{\Gamma}(L_K(E))$ are isomorphic as pre-ordered $\Gamma$-monoids. The group completion $G^\Gamma_E$ of $M^\Gamma_E$ is a pre-ordered $\Gamma$-group. If $L_K(E)$ is unital, there is a $\POG_\Gamma^u$-isomorphism of 
$(K^\Gamma_0(L_K(E)), [L_K(E)])$ and $(G_E^\Gamma, [1_E])$ where $[1_E]$ denotes the sum $\sum_{v\in E^0}[v].$

\section{Fullness}\label{section_fullness}

In the rest of the paper, $K$ is any field. We also recall that $\Gamma=\langle t\rangle\cong\Zset$ is the infinite cyclic group generated by $t.$ After a lemma which follows from \cite[Corollary 5.8]{Ara_et_al_Steinberg}, we formulate and prove the fullness result, Theorem \ref{theorem_fullness}. 

\begin{lemma}
Let $E$ be an arbitrary graph. If $P$ and $Q$ are two finitely generated graded projective modules, the relation $[P]=[Q]$ holds in $K_0^\Gamma(L_K(E))$ if and only if $P\cong_{\gr}Q.$ 
\label{lemma_cancellativity}
\end{lemma}
\begin{proof}
By \cite[Corollary 5.8]{Ara_et_al_Steinberg}, the monoid $\V^\Gamma(L_K(E))$ is cancellative.
This implies that $V^\Gamma(L_K(E))=K_0^\Gamma(L_K(E))^+.$ So, the relation $[P]=[Q]$ holds in $K_0^\Gamma(L_K(E))$ if and only if the same relation holds in $\V^\Gamma(L_K(E)).$ By the definition of the monoid  $\V^\Gamma(L_K(E)),$ this implies that the modules $P$ and $Q$ are in the same graded isomorphism class. 
\end{proof}

\begin{theorem} (Fullness)  
Let $E$ and $F$ be graphs with finitely many vertices and let $E$ have countably many edges. For a morphism $f: (K_0^\Gamma(L_K(E)), [L_K(E)])\to (K_0^\Gamma(L_K(F)), [L_K(F)])$ of the category $\POG_\Gamma^u,$ there is a unital graded algebra homomorphism $\phi: L_K(E)\to L_K(F)$ such that $K_0^\Gamma(\phi)=f.$ %
Moreover, if $f$ is injective, then such a map $\phi$ can be found so that it is injective. 
\label{theorem_fullness} 
\end{theorem}
\begin{proof}
For a given $\POG_\Gamma^u$-morphism $f: (K_0^\Gamma(L_K(E)), [L_K(E)])\to (K_0^\Gamma(L_K(F)), [L_K(F)]),$ we define the elements $p_v\in L_K(F)_0$ for $v\in E^0$ and  
$p_e\in L_K(F)_0,$ $x_e\in L_K(F)_1,$ and $y_e\in L_K(F)_{-1}$ for $e\in E^1$ such that the elements $p_v, x_e, y_e$ form an $E$-family. 

{\bf Defining $\mathbf{p_v}$ for $\mathbf{v\in E^0}$.}
Representing $K_0^\Gamma(L_K(E))$ via the finitely generated graded projective right modules, we have that $[L_K(E)]=[\bigoplus_{v\in E^0} vL_K(E)]=\sum_{v\in E^0}[vL_K(E)].$ So, \[[L_K(F)]=f([L_K(E)])=\sum_{v\in E^0}[P_v]=[\bigoplus_{v\in E^0}P_v]\] for some finitely generated graded projective right $L_K(F)$-modules $P_v$ such that $[P_v]=f([vL_K(E)]).$ Without any additional assumptions on $f$, $P_v$ is possibly trivial for some $v\in E^0.$ By Lemma \ref{lemma_cancellativity}, there is a graded isomorphism $L_K(F)\cong_{\gr}\bigoplus_{v\in E^0}P_v.$ Let $Q_v\leq L_K(F)$ be a graded direct summand of $L_K(F)$ such that $Q_v\cong_{\gr} P_v$ by the restriction of the isomorphism $L_K(F)\cong_{\gr}\bigoplus_{v\in E^0}P_v.$ Hence, there are mutually orthogonal homogeneous idempotents $p_v$ in $L_K(F)$ with their sum equal to the identity element of $L_K(F)$ such that $p_vL_K(F)=Q_v\cong_{\gr} P_v.$ So, \[[p_vL_K(F)]=f([vL_K(E)])\] holds in $K_0^\Gamma(L_K(F)).$   

{\bf Defining $\mathbf{p_e}$ for $\mathbf{e\in E^1}.$} Next, we define elements $p_e$ such that $[p_e]=f([ee^*])$ for every $e\in E^1.$ We first do that for edges with regular sources. So, let $v$ be a regular vertex. In this case, we have that $[p_vL_K(F)]=f([vL_K(E)])=\sum_{e\in \so^{-1}(v)}f([ee^*L_K(E)]).$ Thus, there are finitely generated graded projective right $L_K(F)$-modules $P_e$ such that  
\[[p_vL_K(E)]=[\bigoplus_{e\in \so^{-1}(v)} P_e]\] and that $[P_e]=f([ee^*L_K(E)]).$ By Lemma \ref{lemma_cancellativity}, $p_vL_K(F)\cong_{\gr} \bigoplus_{e\in \so^{-1}(v)} P_e.$  
Let $Q_e$ be a direct summand of $p_vL_K(F)$ 
such that $Q_e\cong_{\gr} P_e$ by the restriction of the isomorphism $p_vL_K(F)\cong_{\gr}\bigoplus_{e\in \so^{-1}(v)} P_e.$
Thus, there are mutually orthogonal homogeneous idempotents $p_e\in p_vL_K(F)_0p_v$ such that $p_eL_K(F)=Q_e\cong_{\gr}P_e.$ Hence, \[[p_eL_K(F)]=f([ee^*L_K(E)])\] holds in $K_0^\Gamma(L_K(F)).$  

If $v$ is an infinite emitter, we aim to define homogeneous idempotents $p_e$ and $p_Z^v$ to be the images of $ee^*$ and $v-\sum_{e\in Z}ee^*$ respectively for $e\in \so^{-1}(v)$ and finite and nonempty set $Z\subseteq \so^{-1}(v).$ Let us index the elements of $\so^{-1}(v)$ so that it is equal to $\{e_0, e_1,\ldots\}$ and let $Z_n=\{e_0,\ldots, e_n\}.$ 

For $e=e_0,$ $ee^*v=vee^*=ee^*$ so $ee^*\leq v$ implying that $v-ee^*$ is also a homogeneous idempotent and that $ee^*L_K(E)$ is a direct summand of $vL_K(E)$ with the complement $(v-ee^*)L_K(E)$. Hence, there are finitely generated graded projective right $L_K(F)$-modules $P_e$ and $P_e^v$ such that $[P_e]+[P_e^v]=[p_vL_K(F)].$ Using Lemma \ref{lemma_cancellativity} again, we obtain orthogonal homogeneous idempotents $p_e$ and $p_{\{e\}}^v$ in $p_vL_K(F)p_v$ such that $p_e+p_{\{e\}}^v=p_v.$  Repeating the same argument for $f=e_1$ and for $v-ee^*$ in the place of $v$ (which works since $v-ee^*=ff^*+(v-ee^*-ff^*)$), we obtain orthogonal homogeneous idempotents $p_f$ and $p_{\{e,f\}}^v$ which add up to $p_{\{e\}}^v$ and such that $p_e$ is orthogonal to $p_f.$ Let $p_{\{f\}}^v$ be $p_e+p_{\{e,f\}}^v$ so that we have that $p_v=p_e+p_{\{e\}}^v=p_e+p_f+p_{\{e,f\}}^v=p_f+p_{\{f\}}^v$ and that $p_{\{e\}}^v=p_v-p_e=p_f+p^v_{\{f\}}-p_e=p_f+p_e+p_{\{e,f\}}^v-p_e=p_f+p_{\{e,f\}}^v.$ 

Continuing this argument inductively, we obtain
orthogonal idempotents $p_{e_n}$ and $p_{Z_n}^v$ with $p^v_{Z_{n-1}}$ as their sum. This ensures that $p_{e_n}$ is orthogonal to $p_{e_i}$ for all $i=0, \ldots, n-1.$ Since $p^v_{Z_{n-1}}\leq p_v$ and $p_{e_n}\leq p^v_{Z_{n-1}},$ we have that $p_{e_n}\leq p_v$ so $p_{e_n}p_v=p_vp_{e_n}=p_{e_n}.$  
If $Z\subseteq Z_n$ contains $e_n,$ we define $p_Z^v$ to be $p^v_{Z_n}+\sum_{e\in Z_n-Z} p_e.$ Hence,   
\[p_Z^v+\sum_{e\in Z}p_e=p^v_{Z_n}+\sum_{e\in Z_n-Z}p_e+\sum_{e\in Z}p_e=p_{Z_n}^v+\sum_{e\in Z_n}p_e=p_{Z_{n-1}}^v-p_{e_n}+\sum_{e\in Z_n}p_e=p_{Z_{n-1}}^v+\sum_{e\in Z_{n-1}}p_e=p_v\]
and if $W$ is such that $Z\subsetneq W\subseteq Z_n$ also contains $e_n,$ then
\[p^v_Z=p^v_{Z_n}+\sum_{e\in Z_n-Z}p_e=
p^v_{Z_n}+\sum_{e\in Z_n-W}p_e+\sum_{e\in W-Z}p_e=p_W^v+\sum_{e\in W-Z}p_e.\]
  
{\bf Defining $\mathbf{x_e}$ and $\mathbf{y_e}$ for $\mathbf{e\in E^1}.$} Note that $e$ and $e^*$ realize the equivalence $ee^*\sima_{\gr} e^*e$  for $e\in E^1.$ So, $ee^*L_K(E)\cong_{\gr} (-1)e^*eL_K(E)$ by Lemma \ref{lemma_graded_equivalence}. Hence, in $K_0^\Gamma(L_K(F)),$ we have that \[[p_eL_K(F)]=f([ee^*L_K(E)])=f([(-1)e^*eL_K(E)])=f(t[\ra(e)L_K(E)])=\]\[tf([\ra(e)L_K(E)])=t[p_{\ra(e)}L_K(F)]=[(-1)p_{\ra(e)}L_K(F)]\] for every $e\in E^1.$ 
By Lemma \ref{lemma_cancellativity}, this implies that $p_eL_K(F)\cong_{\gr}(-1)p_{\ra(e)}L_K(F).$ By Lemma \ref{lemma_graded_equivalence}, there are  $x_e\in p_eL_K(F)_1p_{\ra(e)}$ and $y_e\in p_{\ra(e)}L_K(F)_{-1}p_e$ such that $x_ey_e=p_e$ and $y_ex_e=p_{\ra(e)}.$

{\bf Checking the axioms.} We check that 
the elements $p_v, x_e, y_e$ for $v\in E^0, e\in E^1$ form an $E$-family. 
By construction of $p_v$ for $v\in E^0$, the axiom (V) holds. 

The relations $x_e\in p_eL_K(F)_1p_{\ra(e)}, y_e\in p_{\ra(e)}L_K(F)_{-1}p_e$ imply that 
$x_ep_{\ra(e)}=x_e$ and $p_{\ra(e)}y_e=y_e.$ In addition, if $v=\so(e)$ is regular, then 
$p_vx_e=\sum_{f\in \so^{-1}(v)} p_f x_e=\sum_{f\in \so^{-1}(v)} p_f p_ex_e=p_ex_e=x_e$ and, similarly,
$y_ep_v=y_e\sum_{f\in \so^{-1}(v)} p_f=y_ep_e\sum_{f\in \so^{-1}(v)} p_f=y_ep_e=y_e.$ If $v=\so(e)$ is an infinite emitter, then $p_vx_e=p_v(p_ex_e)=(p_vp_e)x_e=p_ex_e=x_e$ and,  similarly,
$y_ep_v=(y_ep_e)p_v=y_e(p_ep_v)=y_ep_e=y_e.$ 
Thus, both (E1) and (E2) hold. 

To check that the axiom (CK1) holds, recall that that $y_ex_e=p_{\ra(e)}.$ If $e\neq f$ are edges with the same source, then $y_ex_f=(y_ep_e)(p_fx_f)=y_e(p_ep_f)x_f=0$ 
since $p_ep_f=0$ by the definition of the elements $p_e$ and $p_f.$ If $e\neq f$ and $\so(e)=v\neq w=\so(f),$ then $y_ex_f=(y_ep_v)(p_wx_f)=y_e(p_vp_w)x_f=0$
since $p_vp_w=0.$ This shows that (CK1) holds. 

The axiom (CK2) holds since $\sum_{e\in \so^{-1}(v)} x_ey_e=\sum_{e\in \so^{-1}(v)} p_e=p_v$ if $v$ is a regular vertex. 

{\bf The final step.}
By the Universal Property of $L_K(E)$ (see section \ref{subsection_LPAs}), there is a unique algebra homomorphism $\phi: L_K(E)\to L_K(F)$ which maps $v\in E^0$ to $p_v,$ $e\in E^1$ to $x_e,$ and $e^*$ to $y_e.$ Since $p_v$ is in $L_K(F)_0,$ $x_e$ is in $L_K(F)_1,$ and $y_e$ is in $L_K(F)_{-1},$ the homomorphism $\phi$ is graded. 
Note also that $\phi$ is unital since \[\phi(1_{L_K(E)})=\phi(\sum_{v\in E^0} v)=\sum_{v\in E^0}p_v=1_{L_K(F)}.\]

As $\phi(v)=p_v$ for any $v$ and, if $v$ is an infinite emitter, $\phi(v-\sum_{e\in Z}ee^*)=p_Z^v,$ $K_0^\Gamma(\phi)$ and $f$ are equal on the generators of $K_0^\Gamma(L_K(E)).$ Hence, $K_0^\Gamma(\phi)=f.$

If $f$ is injective, the condition $[vL_K(E)]\neq 0$ ensures that $[p_vL_K(F)]\neq 0$ which implies that $p_vL_K(F)\neq 0.$ Thus, $\phi(v)=p_v\neq 0$ for every $v\in E^0.$ By the Graded Uniqueness Theorem (see section \ref{subsection_LPAs}), $\phi$ is injective. 
\end{proof}

Theorem \ref{theorem_fullness} has the following corollary. 

\begin{corollary} 
If $E$ and $F$ are finite graphs and $f:(K_0^\Gamma(L_K(E)), [L_K(E)])\to (K_0^\Gamma(L_K(F)),$ $[L_K(F)])$ is an isomorphism of $\POG^u_\Gamma,$ then there are unital graded algebra monomorphisms $\phi: L_K(E)\to L_K(F)$ and $\psi: L_K(F)\to L_K(E)$ such that $K_0^\Gamma(\phi)=f$ and $K_0^\Gamma(\psi)=f^{-1}.$ 
\label{corollary_of_fullness}
\end{corollary}

\section{Weak faithfulness}\label{section_faithfullness}

Theorem \ref{theorem_faithfulness} specifies the level of faithfulness of the functor $K_0^\Gamma.$ Using vertices and edges of the graph instead of the standard matrix units, the proof generalizes the standard argument for faithfulness of $K_0$ for matricial algebras over a field when the homomorphisms are considered modulo inner automorphisms. 

\begin{theorem} (Weak faithfulness)
If $E$ and $F$ are any graphs and if $\phi, \psi: L_K(E)\to L_K(F)$ are graded algebra homomorphisms (not necessarily unital), then the following are equivalent. 
\begin{enumerate}[\upshape(1)]
\item $K^\Gamma_0(\phi)=K^\Gamma_0(\psi).$

\item The relations $\phi(v)\sima \psi(v)$ and $\phi(ee^*)\sima \psi(ee^*)$ hold in $L_K(F)_0$ for every $v\in E^0$ and every $e\in E^1.$    
\end{enumerate}

If $E$ and $F$ are finite, conditions (1) and (2) are equivalent with each of the following conditions. 
\begin{enumerate}
\item[{\em (3)}] The relations $\phi(v)\sima \psi(v)$ hold in $L_K(F)_0$ for every $v\in E^0.$   
\item[{\em (4)}] There exist elements $x,y\in L_K(F)_0$ such that $xy=\phi(1_{L_K(E)}),$ $yx=\psi(1_{L_K(E)}),$  $x\psi(v)y=\phi(v)$ for every $v\in E^0$. 
\item[{\em (4$^+$)}] There exist elements $x,y\in L_K(F)_0$ such that $xy=\phi(1_{L_K(E)}),$ $yx=\psi(1_{L_K(E)}),$  $x\psi(v)y=\phi(v)$ for every $v\in E^0$ and $x\psi(ee^*)y=\phi(ee^*)$ for every $e\in E^1$. 
\item[{\em (5$^+$)}] There is an invertible element $z$ of $L_K(F)_0$ such that $z\psi(v)z^{-1}=\phi(v)$ for every $v\in E^0$ and $z\psi(ee^*)z^{-1}=\phi(ee^*)$ for every $e\in E^1.$

\item[{\em (5)}] There is an invertible element $z$ of $L_K(F)_0$ such that $z\psi(v)z^{-1}=\phi(v)$ for every $v\in E^0.$       
\end{enumerate}
\label{theorem_faithfulness} 
\end{theorem}
\begin{proof} 
For brevity, let $L_K(F)=S.$

Having Lemma \ref{lemma_cancellativity}, the implication (1) $\Rightarrow$ (2) is rather direct. If (1) holds and $v\in E^0,$ then 
\[[\phi(v)L_K(F)]=K^\Gamma_0(\phi)([vL_K(E)])=K^\Gamma_0(\psi)([vL_K(E)])=[\psi(v)L_K(F)],\] so that
$\phi(v)L_K(F)\cong_{\gr}\psi(v)L_K(F)$
holds by Lemma \ref{lemma_cancellativity}. Thus, $\phi(v)\sima_{S_0} \psi(v)$ holds by Lemma \ref{lemma_graded_equivalence}. Similarly, for every $e\in E^1,$
$[\phi(ee^*)L_K(F)]=K^\Gamma_0(\phi)([ee^*L_K(E)])=K^\Gamma_0(\psi)([ee^*L_K(E)])=[\psi(ee^*)L_K(F)]$ which implies $\phi(ee^*)L_K(F)\cong_{\gr}\psi(ee^*)L_K(F)$ by Lemma \ref{lemma_cancellativity} so that  
$\phi(ee^*)\sima_{S_0} \psi(ee^*)$ holds by Lemma \ref{lemma_graded_equivalence}. This shows that (2) holds. 

The converse (2) $\Rightarrow$ (1) holds since (2) implies that the maps $K^\Gamma_0(\phi)$ and $K^\Gamma_0(\psi)$ are equal on the generators of $K_0^\Gamma(L_K(E)).$ 

Let us assume that $E$ and $F$ are finite  graphs now. Condition (2) trivially implies (3) and (3) implies (1) by the same argument as for (2) $\Rightarrow$ (1).   
To finish the proof, we show that $(2) \Rightarrow (4^+)\Rightarrow (5^+)\Rightarrow (5)\Rightarrow (3).$ The proof of $(4^+)\Rightarrow (5^+)$  also shows that (4) $\Rightarrow$ (5), so this also shows that $(2) \Rightarrow (4)\Rightarrow (5)\Rightarrow (3).$  

Assume that (2) holds. For every $e\in E^1,$ let $x_e\in \phi(ee^*)S_0\psi(ee^*)$ and $ y_e\in \psi(ee^*)S_0\phi(ee^*)$ be such that $x_ey_e=\phi(ee^*)$ and $ y_ex_e=\psi(ee^*).$
If $e\neq f,$ $y_ex_f=y_e\phi(ee^*)\phi(ff^*)x_f=y_e\phi(ee^*ff^*)x_f=y_e0x_f=0.$ Similarly, $x_ey_f=0.$

For $v\in E^0$ regular, let $x_v=\sum_{e\in \so^{-1}(v)} x_e$ and $y_v=\sum_{e\in \so^{-1}(v)} y_e.$ We claim that $x_v$ and $y_v$ realize the equivalence $\phi(v)\sima_{S_0}\psi(v).$ Indeed,
\[x_vy_v=\sum_{e\in \so^{-1}(v)}\sum_{f\in \so^{-1}(v)} x_e y_f=\sum_{e\in \so^{-1}(v)} x_ey_e=\sum_{e\in \so^{-1}(v)} \phi(ee^*)=\phi(\sum_{e\in \so^{-1}(v)}ee^*)= \phi(v).\] Analogously, $y_vx_v=\psi(v).$ Note that $\phi(v)x_v=\sum_{e\in \so^{-1}(v)}\phi(ee^*)\sum_{f
\in \so^{-1}(v)}x_f=\sum_{e\in \so^{-1}(v)}\phi(ee^*)x_e=\sum_{e\in \so^{-1}(v)}x_e=x_v.$ By analogous arguments, $x_v\in S_0\psi(v)$ and $y_v\in \psi(v)S_0\phi(v).$

If $v$ is a sink, we let $x_v\in \phi(v)S_0\psi(v)$ and $y_v\in \psi(v)S_0\phi(v)$ be any two elements realizing the equivalence $\phi(v)\sima_{S_0} \psi(v).$ 

If $v$ and $w$ are two different vertices of $E,$ we have that 
\[x_vy_w=x_v\psi(v)\psi(w)y_v=x_v\psi(vw)y_v=x_v0y_v=0.\] Similarly, $y_vx_w=0.$

Let $x=\sum_{v\in E^0}x_v$ and $y=\sum_{v\in E_0}y_v.$ Then 
\[
xy = \sum_{v\in E^0} x_v\sum_{w\in E^0} y_w= \sum_{v\in E^0}x_vy_v= \sum_{v\in E^0} \phi(v)
=\phi(\sum_{v\in E^0}v)=\phi(1_{L_K(E)})
\]
and, similarly, $yx=\psi(1_{L_K(E)}).$ 
 
For $v\in E^0,$ we have that 
\[
x\psi(v)y = 
\sum_{w\in E^0}x_w \psi(v) \sum_{w\in E^0}y_w=x_v\psi(v)y_v=x_vy_v=\phi(v).\]

For any $e\in E^1,$ we have that 
\[
x\psi(ee^*)y = 
\sum_{w\in E^0}x_w \psi(ee^*) \sum_{w\in E^0}y_w=x_{\so(e)}\psi(ee^*)y_{\so(e)}=\]\[
\sum_{f\in \so^{-1}(\so(e)))}x_f \psi(ee^*) \sum_{f\in \so^{-1}(\so(e)))}y_f=x_e\psi(ee^*)y_e=\phi(ee^*).\]
This shows that (4$^+$) holds.

Assume that (4$^+$) holds and let us show (5$^+$). For brevity, let $1_E=1_{L_K(E)}$ and $1_F=1_{L_K(F)}.$ By replacing $x$ with $\phi(1_E)x\psi(1_E)$ and $y$ with $\psi(1_E)y\phi(1_E),$ we can assume that $x\in \phi(1_E)S_0\psi(1_E)$ and $y\in \psi(1_E)S_0\phi(1_E)$ realize the equivalence $\phi(1_E)\sima_{S_0}\psi(1_E).$ Since $F$ is a finite graph, $S_0$ is isomorphic to an ultramatricial algebra over $K$ with unital connecting maps (\cite[Corollary 2.1.16]{LPA_book}). Hence, $S_0$ is a unit-regular ring. By \cite[Theorem 4.1]{Goodearl_book}, if two idempotents $p$ and $q$ of a unit-regular ring are algebraically equivalent then $1-p$ and $1-q$ are also algebraically equivalent. 
Thus, the equivalence $\phi(1_E)\sima_{S_0}\psi(1_E)$ implies that $1_F-\phi(1_E)\sima_{S_0}1_F-\psi(1_E).$ %
If $x'\in (1_F-\phi(1_E))S_0(1_F-\psi(1_E))$ and $y'\in (1_F-\psi(1_E))S_0(1_F-\phi(1_E))$ are some elements realizing this equivalence, one checks that $z=x+x'$ and $w=y+y'$ are mutually inverse elements such that $z\psi(v)w =\phi(v)$ for every $v\in E^0,$ and $z\psi(ee^*)w=\phi(ee^*)$ for every $e\in E^1.$

(5$^+$) trivially implies (5) and (5) $\Rightarrow$ (3) holds by a direct argument: 
if (5) holds and $z\in S_0$ is as in condition (5), then $x_v=z\psi(v)$ and $y_v=z^{-1}$
realize the algebraic equivalence $\phi(v)\sima_{S_0}\psi(v).$ 
\end{proof}

In the case that $E=F$ and $E$ is finite, $\phi$ is a graded automorphism, and $\psi$ is the identity, \cite[Proposition 6.3]{Ara_Pardo_graded_K_classification} states that condition (1) of Theorem \ref{theorem_faithfulness} is equivalent with the requirement that there is an inner automorphism $\theta$ of $L_K(E)_0$ such that $\phi$ and $\theta$ coincide on a nonempty and finite subset of $L_K(E)_0.$ 

If $E$ and $F$ are graphs without sources and if the maps $\phi$ and $\psi$ as in the assumption of Theorem \ref{theorem_faithfulness} are also graded {\em isomorphisms} such that condition (1) of Theorem \ref{theorem_faithfulness} holds, then \cite[Theorem 6.5]{Ara_Pardo_graded_K_classification} exhibits an 
injective graded endomorphism $\theta$ of $L_K(F)$ such that the restriction of $\theta$ to $F^0$ is the conjugation by an invertible element of $L_K(F)_0$ and such that $\theta\psi=\phi.$ A necessary and sufficient condition for such $\theta$ to be onto is also given. 

Next, we exhibit an example showing that one cannot strengthen condition (5) of Theorem \ref{theorem_faithfulness} by requiring that $\phi(e)=z\psi(e)z^{-1}$ holds for all edges $e$ of $E$. This example also implies that $K_0^\Gamma$ is not faithful on $\LE$ (note that this also follows from \cite[Example 6.7]{Ara_Pardo_graded_K_classification}).

\begin{example}
Let $E$ be the graph $\xymatrix{\bullet^u\ar@(ul,ur)^e\ar@(dl, dr)_f  \ar[r]^g&\bullet^v  },$ let $\phi$ be the identity map on $L_K(E)$, and let $\psi$ map the elements $u,v, g$ and $g^*$ identically to themselves and map $e$ to $f,$ $f$ to $e,$ $e^*$ to $f^*,$ and $f^*$ to $e^*.$ One checks that the elements of the image of $\psi$ constitute an $E$-family, so $\psi$ extends to a ring homomorphism which we also call $\psi$ and which is graded since the degrees of the $E$-family are adequate. The maps $\phi$ and $\psi$ satisfy condition (3) of Theorem \ref{theorem_faithfulness}, so they induce the same map on $K_0^\Gamma(L_K(E)).$ The unitary and selfadjoint element $w=ef^*+fe^*+gg^*+v$ is such that conditions (5) and (5$^+$) hold. Note that $we\neq fw$ (assuming that $we=fw$ implies $e=f$ which is a contradiction).

We show that there is no invertible $z\in L_K(E)_0$ such that $\theta\psi=\phi$ where $\theta$ is the conjugation by $z.$  Assume, on the contrary, that there is an invertible element $z\in L_K(E)_0$ such that $z\psi(r)z^{-1}=r$ holds for every $r\in L_K(E).$ Let 
\[z=kv+k_ggg^*+\sum_{i=1}^nk_ip_igg^*q_i^*+\sum_{j=1}^m k_j r_js_j^*\]
where the paths $p_i, q_i, r_j,$ and $s_j$ consist only of edges $e$ and $f,$ $|p_i|=|q_i|>0,$ $|r_j|=|s_j|>0,$ and $k, k_g, k_i, k_j\in K$ for every $i=1,\ldots, n$ and every $j=1,\ldots, m.$ Note that $uz=z-kv.$

Let $p$ be any path consisting only of edges $e$ and $f$ which has length larger than any of $p_i, q_j, r_j,$ and $s_j.$ Let $q=\psi(p).$ 
Since $pz=zq,$ we have that $p^*pz=p^*zq,$ so $z-kv=uz=p^*zq.$ As $p^*v=p^*g=p^*p_ig=0$ for any $i=1,\ldots, n,$ we have that   
\[z-kv=\sum_{j=1}^m k_jp^*r_js_j^*q.\]
We claim that the sum on the right side is zero. 
If $p^*r_js_j^*q\neq 0,$ then $r_j$ is a prefix of $p,$ necessarily proper by the choice of $p,$ and $s_j$ is a proper prefix of $q.$ If $p=r_jp'$ and $q=s_jq'$ for some paths $p'$ and $q',$ then $q'=\psi(p')$ (also $s_j=\psi(r_j)$). Thus,  
$p^*r_js_j^*q=(p')^*r_j^*r_js_j^*s_j\psi(p')=(p')^*\psi(p')=0$ since $p'$ and $\psi(p')$ have the first edge different. This contradicts the assumption that $p^*r_js_j^*q\neq 0.$ So, we have that $z-kv=0.$ This is a contradiction since $kv$ is not invertible ($kvu=0,$ for example).
\label{example_not_faithful} 
\end{example}

\section{Some reflections on the results}
\label{section_reflections}

The progress made by Theorem \ref{theorem_fullness} and Corollary \ref{corollary_of_fullness} does not settle the question whether the GCC holds. We list some other questions and thoughts.  

\subsection{Removing the cardinality assumptions}
Let $E$ and $F$ have infinitely many vertices
and the standard generating intervals of $L_K(E)$ and $L_K(F)$ be considered instead of the standard order-units. If the assumption of Theorem \ref{theorem_fullness} is modified accordingly, the question is whether the conclusion of Theorem \ref{theorem_fullness} holds without the requirement that the graded algebra map is unital. 

\subsection{The level of constructiveness of Theorem \ref{theorem_fullness}}\label{subsection_constructiveness}
Another natural question is whether one can  explicitly produce a graded homomorphism $\phi$ as in Theorem \ref{theorem_fullness}. The proof of Theorem \ref{theorem_fullness} has several non-constructive steps present mainly because
a relation $[P]=[Q]$ on the $K_0^\Gamma$-level does not produce a specific graded isomorphism $P\cong_{\gr} Q$ on the algebra level. Such isomorphism would be needed if one is to explicitly produce $p_v$ for a module $P_v$ in the first step of the proof of Theorem \ref{theorem_fullness}. 
The same non-constructive step is present in the subsequent stage of the proof when defining $p_e.$ Similarly, when defining $x_e$ and $y_e,$ one would need to know a specific graded isomorphism $p_eL_K(F)\cong_{\gr}(-1)p_{\ra(e)}L_K(F)$ in order to come up with the elements $x_e$ and $y_e$ realizing the equivalence $p_e\sima_{\gr} p_{\ra(e)}$. So, one can say that the proof of Theorem \ref{theorem_fullness} is not constructive.

\subsection{A relation of the GCC with symbolic dynamics}
Let $E$ and $F$ be finite graphs without sinks, let $A_E$ and $A_F$ be their incidence matrices, and let $G_E^\Gamma$ and $G_F^\Gamma$ be the graph $\Gamma$-groups of $E$ and $F$ respectively. 
Recall that $L_K(E)$ and $L_K(F)$ are graded Morita equivalent if there is an equivalence of the categories of graded right modules of $L_K(E)$ and $L_K(F)$ which commutes with the shift functor $\tau_n$ (given by $\tau_n(M)=(n)M$ for all $n\in \Zset).$
If the shift equivalence and the strong shift equivalence in the diagram below refer to such equivalences over $\Zset^+,$ then the two implications in the second row of the diagram below hold by \cite[Proposition 15]{Roozbeh_Dynamics} (the first also by \cite[Theorem 3.12]{Ara_Pardo_graded_K_classification}). The other three implications in the diagram are direct and the equivalence in the second row holds by \cite[Corollary 12]{Roozbeh_Dynamics} (if $E$ and $F$ have no sources, also by \cite[Theorem 3.10]{Ara_Pardo_graded_K_classification}). 

{\tiny
\begin{center}
\begin{tabular}{ccccccc}
 & &
\begin{tabular}{|c|}\hline
$L_K(E)\cong_{\gr}L_K(F)$ \\ \hline 
\end{tabular}& $\Rightarrow$ & 
\begin{tabular}{|c|}\hline
$(G^\Gamma_E, 1_E)\cong (G^\Gamma_F, 1_F)$
\\ \hline 
\end{tabular} & &
\\
&&$\Downarrow$&&$\Downarrow$&&\\
\begin{tabular}{|c|}\hline
$A_E$ and $A_F$ are\\ 
strongly shift equiv. +\\  $E$ and $F$ have no sources.\\ \hline
\end{tabular} & $\Rightarrow$ &
\begin{tabular}{|c|}\hline
$L_K(E)$ and $L_K(F)$ are \\
graded Morita equivalent.\\ \hline
\end{tabular}
& $\Rightarrow$ & 
\begin{tabular}{|c|}\hline
$G^\Gamma_E\cong G^\Gamma_F$ \\ \hline
\end{tabular} & $\Leftrightarrow$ &\begin{tabular}{|c|}\hline
$A_E$ and $A_F$ are \\
shift equivalent.\\ \hline
\end{tabular}
\end{tabular}
\end{center}}

By \cite[Example 18]{Roozbeh_Dynamics}, there 
is a finite graph $E$ without sources or sinks  such that $A_E$ and $A^t_E$ are strongly shift equivalent. If $F$ is the opposite graph of $E$ (obtained by reversing the edges of $E$), then $A_F=A^t_E.$ 
However, by the same example, no order-preserving map $G_E^\Gamma\to G_F^\Gamma$ can also be order-unit-preserving. In particular, this shows that $L_K(E)$ and $L_K(F)$ are not  graded isomorphic. Hence, the vertical implications are strict. 

By the Kim-Roush example from \cite{Kim_Roush}, there are $7\times 7$ matrices $A$ and $B$ with their entries in $\Zset^+$ such that $A$ and $B$ are shift equivalent but not strongly shift equivalent. If $E$ and $F$ are graphs such that their incidence matrices are $A$ and $B$ respectively, then $E$ and $F$ have no sources or sinks. This shows that the composition of the two implications in the second row is strict. So, at least one of these two implications is strict, but it is not clear which one, possibly both. The statement that the second implication is not strict for all finite graphs is sometimes also referred to as the GCC.

We also note that, by \cite[Theorem A]{Eilers_et_al}, the four conditions in the middle two columns of the above diagram are all equivalent for countable graphs with the property that if there is an edge from a vertex $v$ to a vertex $w,$ then there are infinitely many edges from $v$ to $w$.


\begin{thebibliography}{10}
\bibitem{LPA_book} G. Abrams, P. Ara, M. Siles Molina, Leavitt path algebras, Lecture Notes in Mathematics 2191, Springer, London, 2017.

\bibitem{Ara_et_al_Steinberg} P. Ara, R. Hazrat, H. Li, A. Sims, \emph{Graded Steinberg algebras and their representations}, Algebra Number Theory {\bf 12 (1)} (2018), 131--172.

\bibitem{Ara_Pardo_graded_K_classification} P. Ara, E. Pardo, \emph{Towards a $K$-theoretic characterization of graded isomorphisms between Leavitt path algebras}, J. $K$-Theory {\bf 14} (2014), 203 -- 245. 

\bibitem{Arnone} G. Arnone, \emph{Lifting morphisms between graded Grothendieck groups of Leavitt path algebras}, https://doi.org/10.48550/arXiv.2206.06759. 

\bibitem{Eilers_et_al} S. Eilers, E. Ruiz, A. Sims, \emph{Amplified graph $C^*$-algebras II: reconstruction}, Proc. Amer. Math. Soc. Ser. B {\bf 9} (2022), 297--310. 

\bibitem{Goodearl_book} K. R. Goodearl, von Neumann regular rings, 2nd ed., Krieger Publishing Co., Malabar, FL, 1991.

\bibitem{Roozbeh_Annalen} R. Hazrat, \emph{The graded Grothendieck group and classification of Leavitt path algebras,} Math. Annalen {\bf 355 (1)} (2013), 273--325.

\bibitem{Roozbeh_Dynamics} R. Hazrat, \emph{ The dynamics of Leavitt path algebras}, J. Algebra {\bf 384} (2013), 242 -- 266.

\bibitem{Roozbeh_book} R. Hazrat, Graded rings and graded Grothendieck groups, London Math. Soc. Lecture Note Ser. 435, Cambridge Univ. Press, 2016.

\bibitem{Roozbeh_Lia_Ultramatricial}  R. Hazrat, L. Va\v s, \emph{$K$-theory classification of graded ultramatricial algebras with involution}, Forum Math., {\bf 31 (2)} (2019), 419--463. 

\bibitem{Roozbeh_Lia_comparability} R. Hazrat, L. Va\v s, \emph{Comparability in the graph monoid}, New York J. Math. {\bf 26} (2020), 1375--1421. 

\bibitem{Kim_Roush} K. H. Kim, F. W. Roush, \emph{The Williams conjecture is false for irreducible subshifts}, Ann. of Math. {\bf (2) 149} (1999), 545 -- 558. 

\bibitem{Lia_graded_UR} L. Va\v s, \emph{Graded cancellation properties of graded rings and graded unit-regular Leavitt path algebras}, Algebr. Represent. Theory {\bf 24 (3)} (2021), 625--649.

\bibitem{Lia_realization} L. Va\v s, \emph{Simplicial and dimension groups with group action and their realization}, Forum Math. {\bf 34 (3)} (2022), 565--604. 

\end{thebibliography}
\end{document}